\documentclass[11pt]{article}
\usepackage{latexsym}
\usepackage{amsthm}
\usepackage{amssymb}
\usepackage{amsmath}
\usepackage{rotating}
\usepackage{multirow}
\usepackage{mathrsfs}
\usepackage{wasysym}
\usepackage{enumerate}
\usepackage{esint}
\usepackage{rawfonts}
\input{prepictex}
\input{pictex}
\input{postpictex}

\DeclareMathOperator{\Diff}{\zap{Diff}}
\usepackage[OT2,OT1]{fontenc}
\def\cyr{%
\renewcommand\rmdefault{wncyr}%
\renewcommand\sfdefault{wncyss}%
\renewcommand\encodingdefault{OT2}%
\normalfont
\selectfont}
\DeclareMathAlphabet{\zap}{OT1}{pzc}{m}{it}
\DeclareTextFontCommand{\textcyr}{\cyr}
\def\be{\begin{equation}}
\def\ee{\end{equation}}
\def\bea{\begin{eqnarray*}}
\def\eea{\end{eqnarray*}}

\def\CC{\mathbb C}

\newtheorem{main}{Theorem}

\DeclareMathOperator{\trace}{trace}
\DeclareMathOperator{\Hom}{Hom}

\newtheorem{thm}{Theorem}[section]
\newtheorem{lem}{Lemma}
\newtheorem{prop}[thm]{Proposition}
\newtheorem*{cor}{Corollary}

\def\ZZ{{\mathbb Z}}
\def\RR{{\mathbb R}}
\def\CP{{\mathbb C \mathbb P}}
\begin{document}

\title{Einstein Manifolds, Self-Dual Weyl Curvature, and Conformally K\"ahler Geometry}

\author{Claude LeBrun\thanks{Supported in part by   NSF grant DMS-1906267  and a Simons Fellowship. }\\Stony Brook University}

\date{}
\maketitle

\begin{abstract}
Peng Wu \cite{pengwu} recently announced a beautiful  characterization of conformally K\"ahler, Einstein metrics 
of positive scalar curvature on compact oriented $4$-manifolds via the condition  $\det (W^+) > 0$. 
In this note, we  buttress his claim by providing an entirely different proof of his result.
We then   present further consequences of  our method, which   builds  on techniques previously developed in  
 \cite{lebcake}. 
\end{abstract}

\section{Introduction}

Recall that a Riemannian metric $h$ is said to be {\em Einstein} if it has constant Ricci curvature. This is 
equivalent to saying that it solves the {\em Einstein equation}
\begin{equation}
\label{einstein}
r= \lambda h,
\end{equation}
where $r$ is the Ricci tensor of $h$,  and where the  real constant $\lambda$  (which is  not specified in advance)   is called the {\em Einstein constant} of $h$.
Given a smooth compact manifold $M$, it is a  fundamental  problem of modern Riemannian geometry  to completely understand the  moduli
space
$$\mathscr{E}(M) = \{ \mbox{Solutions of } \eqref{einstein}\}/(\Diff (M) \times \RR^\times),$$
of Einstein metrics on $M$, where the diffeomorphism group $\Diff (M)$ of course acts on solutions of \eqref{einstein} via pull-backs, while the  group of positive reals $\RR^+$ 
acts on solutions by constant rescalings. One key goal  of this  paper is  to    study this problem for  a  specific class of  $4$-manifolds $M$.  

Our focus on dimension four  reflects the degree to which this dimension  seems to represent a sort of ``Goldilocks zone''  for the Einstein equation  \eqref{einstein}. 
In lower dimensions, Einstein metrics necessarily have constant sectional curvature, making them  locally  boring --- albeit  still globally interesting. 
In higher dimensions, on the other hand, 
Einstein metrics  turn out to exist in surprising profusion, leading to  wildly disconnected Einstein moduli spaces   on  even the most familiar manifolds \cite{bohm,bgk,wazi}. 
 But, by contrast, dimension four seems ``just right'' for \eqref{einstein}, 
 because   four-dimensional Einstein metrics display  such a remarkably well-balanced  combination of local
 flexibility and global rigidity that  their geometry  often 
 seems to be optimally adapted to   the manifold where they reside. 
 For example, if 
  $M^4$ is a compact quotient of real or complex-hyperbolic space, or a $4$-torus, or $K3$, 
then the  Einstein  moduli space  $\mathscr{E}(M)$ is actually explicitly known, and   in each case actually  turns out  to  be
 {connected} \cite{bergb,bcg,hit,lmo}.

 Unfortunately, however, there are very few other $4$-manifolds $M$ whose  Einstein moduli spaces $\mathscr{E}(M)$ are both 
 non-empty and completely understood.  
 In particular, we still only partially understand  the Einstein moduli spaces of  
 the  smooth compact  $4$-manifolds  that arise  as {\sf del Pezzo surfaces}. 
 Recall that a compact complex $2$-manifold $(M^4,J)$  is called a  {del Pezzo surface} iff it has ample anti-canonical 
line bundle.   Up to diffeomorphism, there are
exactly ten such manifolds, namely  $S^2 \times S^2$ and the nine connected sums $\CP_2\# m\overline{\CP}_2$,  $m = 0, 1, \ldots, 8$. 
These are exactly \cite{chenlebweb}  the   oriented smooth compact  
$4$-manifolds that  admit both an Einstein metric with $\lambda > 0$ and an orientation-compatible symplectic structure. However, on any of these spaces, 
every known Einstein metric  is conformally K\"ahler. 
In most cases,   these currently-known Einstein metrics  are actually K\"ahler-Einstein \cite{sunspot,tian},
but in exactly two cases  they are instead constructed \cite{chenlebweb,lebhem10} as non-trivial conformal rescalings of
extremal K\"ahler metrics. This situation has  prompted the author to elsewhere characterize  the known Einstein metrics on 
del Pezzo surfaces by means of two different non-K\"ahler criteria. First, they are \cite{lebuniq} the only $\lambda >0$ Einstein metrics on compact $4$-manifolds
that are Hermitian with respect to an integrable complex structure. Perhaps more compellingly,  they are also the only Einstein metrics on compact  oriented $4$-manifolds
for which the self-dual Weyl curvature $W^+$ is everywhere positive in the direction of a global self-dual harmonic $2$-form \cite{lebcake}. 
Because the latter characterization merely depends on the Einstein metric belonging to an explicit open set in the space of Riemannian metrics,
it in particular  allows one to prove that,  on  any del Pezzo $M^4$,  the known Einstein metrics exactly
sweep out  a single  connected component in the 
Einstein moduli space $\mathscr{E}(M)$.

Still, both of  these previous characterizations suffer from the defect  of not being formulated in terms of  a purely local condition on the curvature tensor. 
It is for this reason that  a new characterization recently announced by Peng Wu \cite{pengwu}, formulated purely in terms of a property of the self-dual Weyl curvature, 
 represents an important advance in the subject.

To explain  Wu's criterion, let us first recall that   the bundle  $\Lambda^2$ of 2-forms over an
oriented Riemannian $4$-manifold $(M,h)$  naturally    decomposes, in a conformally  invariant
way, 
as a direct sum
$$
\Lambda^2 = \Lambda^+ \oplus \Lambda^-  
$$
of  the   $(\pm 1)$-eigenspaces $\Lambda^\pm$ of 
the Hodge star 
operator. Here, sections of $\Lambda^+$ are called 
self-dual $2$-forms, while sections of 
$\Lambda^-$ are  called 
anti-self-dual $2$-forms. But since the  Riemann curvature tensor may be identified with 
 a self-adjoint linear map 
$${\mathcal R} : \Lambda^2 \to \Lambda^2$$
it can therefore be decomposed into  irreducible  pieces 
$$
{\mathcal R}=
\left(
\mbox{
\begin{tabular}{c|c}
&\\
$W^++\frac{s}{12}I$&$\stackrel{\circ}{r}$\\ &\\
\cline{1-2}&\\
$\stackrel{\circ}{r}$ & $W^-+\frac{s}{12}I$\\&\\
\end{tabular}
} \right) 
$$
where  $s$ is the {scalar curvature},
$\stackrel{\circ}{r}=r-\frac{s}{4}g$ is  the   
     trace-free Ricci curvature, and where 
$W^\pm$ are the trace-free pieces of the appropriate blocks.
The corresponding pieces  ${W^{\pm a}}_{bcd}$ of the Riemann curvature tensor 
are both conformally invariant, and 
 are  respectively called the
{\em self-dual} and {\em anti-self-dual Weyl curvature} tensors. 

Wu observes that the self-dual Weyl curvature $W^+:\Lambda^+ \to \Lambda^+$ 
of 
any   conformally K\"ahler, Einstein metric on any del Pezzo surface  
satisfies $\det (W^+) > 0$. He then offers a rather terse and cryptic proof 
that the converse is also true. One main purpose of this article
is to provide an entirely different proof of Wu's beautiful result:

\begin{main} 
\label{whoo}
Let $(M,h)$ be a  simply-connected compact
 oriented Einstein $4$-manifold, and suppose that its self-dual Weyl curvature $W^+: \Lambda^+\to \Lambda^+$ 
  satisfies $\det (W^+) > 0$ at every point
of $M$. Then $h$ is conformal to an 
  orientation-compatible extremal K\"ahler metric $g$ on $M$.  
   \end{main}

With \cite{lebhem} and 
\cite{lebcake}, this now immediately implies the following:
   
  \begin{cor} Any  simply-connected compact
 oriented Einstein $4$-manifold   with $\det (W^+) > 0$  is orientedly diffeomorphic to a del Pezzo surface. 
 Conversely, the underlying smooth oriented 
 $4$-manifold $M$ of  any  del Pezzo surface carries Einstein metrics $h$ with $\det (W^+)> 0$, 
 and these  sweep out exactly  one connected component of the moduli space $\mathscr{E}(M)$ of Einstein metrics on $M$.
  \end{cor} 
   
    Note that the simple-connectivity hypothesis  is essential in Theorem \ref{whoo}. 
    Otherwise,  a counter-example would be given by 
    $(S^2 \times S^2)/\ZZ_2$,
obtained by dividing  the Riemannian product of two round, unit-radius  $2$-spheres  
 by  the simultaneous  action of  the antipodal map on both factors. However, Proposition \ref{rule} below
shows that this example is typical, in the following sense: for a compact oriented Einstein manifold with $\det (W^+) >0$, 
the only possible   fundamental groups are $\{ 1\}$ and  $\{\pm 1\}$. Thus,  one can always reduce  to the
simply-connected case by at worst passing to a double cover. 

While the method of proof used here is quite different from Wu's, both approaches are deeply  indebted to the pioneering work of 
Derdzi\'{n}ski \cite{derd}. 
In fact, the method developed here also naturally  yields results about more general 
  $4$-manifolds  with harmonic self-dual Weyl curvature:
   
\begin{main} 
\label{whoa}
Let $(M,h)$ be a compact 
 oriented Riemannian  $4$-manifold whose self-dual Weyl curvature $W^+$ is 
 harmonic,  in the sense that 
  $$\delta W^+:=-\nabla\cdot W^+=0.$$ 
  Suppose  moreover that  $b_+(M)\neq 0$, and that $h$ satisfies   $\det (W^+) > 0$ at every point of $M$. 
 Then $M$ admits an orientation-compatible K\"ahler metric $g$ of scalar curvature $s>0$ such that 
 $h=s^{-2}g$.  \end{main}
 
 Conversely, if $(M^4,g,J)$ is a K\"ahler surface of scalar curvature $s> 0$, 
 Derdzi\'{n}ski discovered  that $h= s^{-2}g$ then satisfies both $\delta W^+=0$ and  $\det (W^+) > 0$. This makes it completely 
 straightforward to classify the smooth compact 
 oriented  $4$-manifolds that   carry metrics $h$ of the type covered  by Theorem \ref{whoa}.
Indeed, 
if a  compact complex surface $(M,J)$  admits 
K\"ahler metrics $g$  with $s> 0$, it   is necessarily rational or ruled \cite{yauruled}, and,  conversely, any rational or ruled surface 
has arbitrarily small deformations that admit such metrics \cite{hitpos,mhsung}. Up to oriented diffeomorphism, the  complete
list of the $4$-manifolds $M$ that admit such metrics $h$ therefore exactly  consists  of  $\CP_2$, $(\Sigma^2 \times S^2) \# k \overline{\CP}_2$, 
 and $\Sigma^2 \varkappa S^2$, where  $\Sigma$ is any compact orientable $2$-manifold, $k$ is any non-negative integer, and 
$\Sigma^2 \varkappa S^2$ is the non-trivial oriented $2$-sphere bundle over $\Sigma$.  Notice, however,  that 
 the moduli space of such metrics  on any  of these manifolds 
is always infinite-dimensional, in marked contrast to the Einstein case.

 We should  also point out that 
 dropping the  $b_+(M)\neq 0$  hypothesis in Theorem \ref{whoa} only changes the story very  slightly. Indeed,  as is shown in Proposition \ref{mule} below, 
 any  compact
 oriented  $(M^4,h)$  with $\delta W^+=0$,  $\det (W^+) > 0$, and $b_+(M)=0$  has a double-cover $\hat{M}\to M$ 
 with $b_+(\hat{M})=1$. Theorem \ref{whoa} therefore  applies to the pull-back of $h$ to this double cover.    

These results are all proved  in \S \ref{gist} below. 
  Finally, in \S \ref{grist},  we then  prove a generalization     that does not explicitly require $\det(W^+)$  to be positive:

\begin{main}
\label{cadenza}
Let $(M,h)$ be a compact 
 oriented Riemannian  $4$-manifold that satisfies $\delta W^+=0$. If 
 $$W^+\neq 0 \quad \mbox{and}\quad  \det (W^+) \geq - \frac{5\sqrt{2}}{21\sqrt{21}}|W^+|^3$$
 everywhere on $M$, then  actually  $\det (W^+) > 0$. Thus, after at worst passing to a double cover $\hat{M}\to M$,
 $h$  becomes conformally K\"ahler,  in the manner described by Theorem \ref{whoa}. 
In particular, if $(M,h)$ is a simply-connected Einstein manifold, it  actually  falls under the purview of  Theorem \ref{whoo}. 
\end{main}

\section{The Proofs of Theorems \ref{whoo} and \ref{whoa}}
\label{gist}

Let  $(M,h)$ be a compact oriented Riemannian  $4$-manifold with
  $\det (W^+) > 0$ everywhere. 
   Since $W^+: \Lambda^+\to \Lambda^+$ is  self-adjoint, we can diagonalize  $W^+$ at any point of $M$ as 
$$W^+ = \left[\begin{array}{ccc}\alpha &  &  \\ &  \beta&  \\ &  & \gamma\end{array}\right], $$
by choosing a suitable orthonormal basis  for $\Lambda^+$; and, after re-ordering our basis if necessary, we may arrange that 
$\alpha \geq \beta \geq \gamma$ at our given point. 
However,  by its very definition, 
   the self-dual Weyl curvature $W^+: \Lambda^+\to \Lambda^+$ 
  automatically satisfies  $\trace (W^+) =0$, and this of course means that 
$$\alpha  +  \beta + \gamma =0.$$
It  thus follows that  $\alpha > 0$ and  $\gamma < 0$ as long as   $W^+\neq 0$ at the point in question. 
We  therefore  immediately see   that   $\det W^+ =  \alpha \beta  \gamma$ always  has the same sign as {\em minus} the 
middle eigenvalue $\beta$. 
Consequently, our assumption that  $\det W^+ > 0$ 
is equivalent  to saying that  {\sf exactly one} of the eigenvalues, namely $\alpha$, is positive 
at each point, while both the other two  are  negative. In particular, the positive eigenvalue $\alpha$ 
always has multiplicity one, 
so that $\alpha : M\to \RR$ is always  the unique positive solution of the  characteristic equation $\det (W^+-\alpha I) =0$, 
and so is  a smooth positive function on $M$.  
 Since the   $\alpha$-eigenspace of $W^+$ is exactly the kernel of $(W^+ - \alpha I) : \Lambda^+\to \Lambda^+$, this eigenspace moreover varies smoothly from point to point. Thus,  our 
 assumption  that $\det W^+ > 0$ 
implies  that  the unique  positive eigenspace of $W^+$  defines a smooth real  line sub-bundle
$L\subset \Lambda^+$. Up to bundle isomorphism, it follows  that $L$ is intrinsically  classified  by 
$$w_1(L) \in H^1 (M, \ZZ_2 ) = \Hom (\pi_1(M), \ZZ_2),$$
and so will necessarily be trivial if $M$ is simply-connected --- or, indeed,  if
$\pi_1(M)$ merely does not contain a subgroup of index $2$.

Since  the  condition $\det (W^+) > 0$ is  conformally invariant,  the above discussion  similarly applies to any  
metric $g= f^{-2} h$ arising by conformal rescaling $h$, using    a smooth positive function  $f: M\to \RR^+$. 
  On the other hand, the endomorphism $W^+: \Lambda^+ \to \Lambda^+$ is explicitly 
  given by 
  $$\varphi_{ab}\longmapsto [W^+(\varphi)]_{cd} := \frac{1}{2} {W^{+ab}}_{cd}~\varphi_{ab},$$
  so constructing it out of the conformal-weight-zero tensor field ${W^{+a}}_{bcd}$
  involves raising an index. Thus, replacing $h$ with $g=f^{-2}h$ rescales the top eigenvalue 
  by a factor of $f^2$:
  $$\alpha_g = f^{2}\alpha_h.$$
  We will henceforth impose the interesting  choice  
  \begin{equation}
\label{shrewd}
\boxed{f= \alpha_h^{-1/3}}
\end{equation}
  of the conformal factor $f$, because this then has the 
  nice property that 
  $$ \alpha_g = f^2\alpha_h= \alpha_h^{1/3}= f^{-1}.$$
  It then follows  that $\alpha := \alpha_g$  satisfies 
  \begin{equation}
\label{cunning}
\alpha f \equiv 1
\end{equation}
for this  preferred  conformal rescaling  $g=f^{-2}h$ of the original metric.

With respect to this conformally  altered metric $g$, there exist, at each point, exactly two self-dual 
$2$-forms $\omega$ which satisfy 
\begin{equation}
\label{charlie}
W^+_g(\omega) = \alpha_g ~\omega , \qquad |\omega|_g^2 = 2.
\end{equation}
Since these both  belong to the real line bundle $L\subset \Lambda^+$, 
and differ by sign, we 
 can  find a global   self-dual $2$-form $\omega$ on $M$ satisfying these requirements
 everywhere if and only if  $L$ is trivial, which  happens  precisely  in the case where  $w_1(L)= 0$. On the other hand, if this class
is non-zero, we can then just  pass to the double cover $\varpi: \hat{M}\to M$ given by the
elements of norm $\sqrt{2}$ in $L$, and we then instead  obtain a tautological global self-dual $2$-form on $\hat{M}$ 
satisfying \eqref{charlie} with respect to the pulled-back metric $\hat{g}= \varpi^*g$. In this case, notice that
the connected Riemannian manifold $(\hat{M},\hat{g})$ admits an isometric involution $\sigma : \hat{M}\to \hat{M}$
induced by scalar multiplication by $-1$ in  $L$, and that this involution  satisfies $\sigma^*\omega = -\omega$ by construction.

Our stipulation  that $|\omega|_g^2=2$ has been imposed so that $\omega$ 
can be put in  the point-wise  normal form 
$$\omega = e^1\wedge e^2 + e^3 \wedge e^4$$
by choosing  an appropriate oriented orthonormal frame at any given point. 
Thus,    whether on $M$ or on $\hat{M}$,  our  global $2$-form $\omega$ 
 will give rise to  a unique orientation-compatible 
almost-complex structure $J$ defined by
$$\omega= g(J\cdot , \cdot ).$$
In other words,  the tensor-field $J$  explicitly obtained from $\omega$ by index-raising
$${J_a}^b= \omega_{ac}g^{cb}$$
with respect to $g$ will then automatically 
satisfy 
$${J_a}^b{J_b}^c = - \delta_a^c,$$
thus making it  a $g$-compatible almost-complex structure on $M$ or   $\hat{M}$.

Our main argument will hinge on a few simple facts about self-dual $2$-forms and the Weyl curvature, starting with the following:
\begin{lem} 
\label{bliss}
Let $(M,h)$ be an oriented  Riemannian $4$-manifold for which  $\det (W^+)> 0$ everywhere.   Also suppose that the top eigenspace $L\subset \Lambda^+$ of
$W^+$ is trivial as  a real line bundle $L\to M$. Let $g=f^{-2}h$ be some conformal rescaling of $h$,   and  let $\omega$ then  be a self-dual $2$-form 
on $M$  that  satisfies \eqref{charlie} everywhere. 
Then 
\begin{equation}
\label{negatron} 
  W^+ ( \nabla^a \omega , \nabla_a\omega) \leq 0 ,
\end{equation}
everywhere, where all terms are to be computed  with respect to $g$. 
\end{lem}
\begin{proof} The covariant derivative   $\nabla \omega$ of $\omega$ belongs to $\Lambda^1 \otimes \omega^\perp \subset\Lambda^1 \otimes  \Lambda^+$ because
$\omega$ has constant norm with respect to $g$. The result therefore follows from the fact that $W^+(\phi , \phi) \leq 0$ for any 
$\phi \in \omega^\perp \subset \Lambda^+$. 
\end{proof} 
Secondly, we will need the following standard algebraic observation:
\begin{lem} At any point $p$ of  an oriented $4$-manfold $(M,g)$, 
\begin{equation}
\label{megatron} 
|W^+|^2 \geq \frac{3}{2} \alpha^2
\end{equation}
where $\alpha = \alpha_g$ is the  the top eigenvalue of $W^+_g$ at $p$.
\end{lem}
\begin{proof}
Because $\trace W^+=0$, 
$$
|W^+|^2 = \alpha^2 + \beta^2 + (-\alpha - \beta)^2 = \frac{3}{2}\alpha^2 +2 (\beta + \frac{1}{2} \alpha)^2  \geq \frac{3}{2}\alpha^2
$$
where $\beta$ is  the middle  eigenvalue of $W^+_g$ at $p$.
\end{proof} 
 Finally, 
 we remind the reader of the Weitzenb\"ock formula 
 \begin{equation}
\label{harmless} 
(d+d^*)^2 \omega = \nabla^*\nabla \omega - 2W^+(\omega) + \frac{s}{3}\omega
\end{equation}
for the Hodge Laplacian on   self-dual $2$-forms.

We are now finally ready to see what all this means when $h$ is an Einstein metric. But   our discussion will actually pertain  to  
the much larger class of 
oriented $4$-manifolds $(M,h)$ which have 
 {\em harmonic self-dual Weyl curvature}, in the sense that 
 \begin{equation}
\label{hsdw}
\delta W^+:= -\nabla \cdot W^+=0.
\end{equation}
 When $h$ is Einstein, \eqref{hsdw} holds  as a consequence of the second Bianchi identity; 
 but the reader should keep in mind that 
 \eqref{hsdw}  is actually much  weaker than the Einstein condition. The reason    \eqref{hsdw} will be so useful for our purposes is that it displays 
 a weighted conformal invariance \cite{pr2} under conformal changes of metric. 
 Namely,  if $h$ satisfies  $\delta W^+=0$, then any conformal rescaling $g=f^{-2}h$ will instead have the property that $\delta (fW^+)=0$. 
This then implies the useful   Weitzenb\"ock formula
\begin{equation}
\label{initio}
0 = \nabla^*\nabla (fW^+)+ \frac{s}{2} fW^+ - 6 fW^+\circ W^+ + 2 f|W^+|^2 I 
\end{equation}
for $fW^+$ with respect to $g$. For other applications of this fact, see  \cite{derd,G1,lebcake}.

\begin{thm} 
\label{summit}
Let $(M,h)$ be a compact oriented Riemannian manifold with $\delta W^+=0$ and $\det (W^+) > 0$. 
Also suppose that the positive eigenspace $L\subset \Lambda^+$ of $W^+$ is trivial as a  real line bundle $L\to M$. Then the 
conformally  rescaled metric  $g=f^{-2}h$  defined by \eqref{shrewd} is an 
 orientation-compatible  K\"ahler metric on $M$.
\end{thm}
\begin{proof} Since $L\to M$ is trivial, we can choose a  global self-dual $2$-form $\omega$ on $M$ which satisfies \eqref{charlie} at every point. 
Always working  henceforth with respect to $g$,  we now  take
 the inner product of \eqref{initio} with 
$\omega\otimes \omega$, and  then integrate on $M$. Integrating by parts, and  using 
\eqref{cunning},  \eqref{negatron},  \eqref{megatron}, and \eqref{harmless},  we then have 
\begin{eqnarray*} 0&=& 
\int_M \Big\langle\Big( \nabla^*\nabla fW^+ + \frac{s}{2} fW^+ - 6 fW^+\circ W^+ + 2 f|W^+|^2 I \Big), \omega \otimes
\omega \Big\rangle ~d\mu_g \\
&=& \int_M \Big[ \langle W^+ , \nabla^*\nabla (\omega\otimes \omega )\rangle  + \frac{s}{2} W^+(\omega , \omega ) 
 - 6 |W^+(\omega)|^2+ 2 |W^+|^2 |\omega |^2 \Big] f~ d\mu_g  \\
&=&  
\int_M \Big[ 
- 2W^+(\nabla_e \omega , \nabla^e \omega )-2W^+(\omega , \nabla^e\nabla_e \omega  ) 
\\&& \hspace{2.2in}   + \frac{s}{2} \alpha |\omega|^2 
 - 6 \alpha^2 |\omega|^2+ 2 |W^+|^2 |\omega |^2 \Big] f~ d\mu_g  \\
 &\geq &  
\int_M \Big[ 
-2\alpha \langle \omega , \nabla^e\nabla_e \omega  \rangle  
+ \frac{s}{2} \alpha |\omega|^2 
 - 6 \alpha^2 |\omega|^2+ 3 \alpha^2 |\omega |^2 \Big] f~ d\mu_g  \\
  &=&  
\int_M \Big[ 
2\langle \omega , \nabla^*\nabla \omega  \rangle  
+ \frac{s}{2}  |\omega|^2 
 - 3 \alpha |\omega|^2 \Big] ~(\alpha f)~ d\mu_g  \\
 &=&  
\int_M \Big[ 
\frac{1}{2}\langle \omega , \nabla^*\nabla \omega  \rangle      + \frac{3}{2}\langle \omega , \nabla^*\nabla \omega  \rangle 
+ \frac{s}{2}  |\omega|^2 
 - 3 W^+(\omega , \omega) \Big] ~ d\mu_g  \\ 
 &=&  \frac{1}{2} \int_M |\nabla\omega |^2 ~ d\mu_g  + \frac{3}{2} 
\int_M \Big\langle \omega , \nabla^*\nabla \omega
- 2 W^+(\omega) + \frac{s}{3}\omega 
 \Big\rangle 
 ~ d\mu_g  \\ 
 &=& 
 \frac{1}{2} \int_M |\nabla\omega |^2 ~ d\mu_g  +
 \frac{3}{2}  \int_M 
\Big\langle \omega , (d+d^*)^2 \omega \Big\rangle 
~ d\mu_g  \\ 
 &=& 
 \frac{1}{2} \int_M |\nabla\omega |^2 ~ d\mu  +
 3 \int_M |d\omega|^2 
 ~d\mu  \\ 
 &\geq& 
 \frac{1}{2} \int_M |\nabla\omega |^2 ~ d\mu  , 
 \end{eqnarray*}
 This shows that  $\nabla \omega \equiv 0$ with respect to our rescaled metric $g$. Since it   of course also follows  that $\nabla J\equiv 0$, 
we now see that  $(M,g,J)$ is actually   a K\"ahler manifold, with K\"ahler form $\omega$. In particular, this shows that 
  the initial metric $h=f^2 g$ is  
    conformally K\"ahler. 
    \end{proof} 

On the other hand,  because the curvature tensor of a K\"ahler surface $(M^4,g,J)$  belongs to $\odot^2 \Lambda^{1,1}$, 
the fact that 
\begin{equation}
\label{decore}
\Lambda^+= \RR \omega \oplus \Re e~\Lambda^{2,0}
\end{equation}
 for $(M,g)$  implies that that its 
 self-dual Weyl curvature 
 takes the form 
$$W^+ = \left[\begin{array}{ccc}s/6 &  &  \\ & -s/12 &  \\ &  & -s/12\end{array}\right]$$
in an orthonormal  basis adapted to \eqref{decore}. In particular, $\det (W^+) = s^3/3^32^5$, so 
 a K\"ahler metric $g$ has $\det (W^+)> 0$ if and only if its scalar curvature $s$ is positive. It follows that any K\"ahler metric
conformal to a Riemannian metric  $h$ with  $\det W^+ > 0$ must necessarily have $s> 0$. Moreover, we now see that  the top eigenvalue $\alpha$ of
$W^+$ for any such metric $g$ is given by $s/6$. Thus,  when $h$  satisfies \eqref{hsdw}, we have succeeded in expressing it  as 
$h=\alpha^{-2}g= 36s^{-2}g$ for a K\"ahler metric $g$ of positive scalar curvature $s$. 
However, when this happens,   $\tilde{g}= 6^{2/3}g$ is  also a K\"ahler metric, and has   scalar curvature $\tilde{s}=6^{-2/3}s$, and we 
therefore  also  have  $h=\tilde{s}^{-2}\tilde{g}$ for a K\"ahler metric $\tilde{g}$ with positive scalar curvature $\tilde{s}$. 
This was the form preferred by Derdzi\'{n}ski \cite{bes,derd}, who discovered that, conversely, 
 any K\"ahler surface $(M^4,g,J)$ of  scalar curvature $s> 0$ 
gives rise to a Riemannian metric $h$ on $M$ with $\delta W^+=0$ and $\det W^+ > 0$ via the ansatz $h= s^{-2}g$.

On the other hand,  any compact K\"ahler surface $(M^4,g,J)$ of positive scalar curvature has  geometric genus 
$h^{2,0}=0$ by an argument due to  Yau \cite{yauruled}. But since $b_+(M) = 1 + 2h^{2,0}$ for any 
compact K\"ahler surface \cite{bpv}, this is equivalent to saying  $b_+(M) = 1$. Geometrically, 
this means that a self-dual $2$-form on $(M,g)$  
is harmonic if and only if it is a constant multiple of the K\"ahler form $\omega$. Of course, since the space of 
self-dual harmonic $2$-forms is conformally invariant in dimension $4$,  we also see that the 
self-dual harmonic $2$-forms on $(M,h)$ likewise consists of the constant multiples of $\omega$. 

This now allows us 
 to back-track a little, and  finally deal with  the case where  the the real-line bundle $L\to M$
is non-trivial. In this setting, the conformal factor defined by \eqref{shrewd} still defines a metric $g$ on $M$, 
but it is only when we pull it back to $\varpi : \hat{M}\to M$ that this rescaled  metric can be associated with a global
self-dual $2$-form 
$\omega$ satisfying  $W^+ (\omega ) =\alpha \omega$ and $|\omega|_{\varpi*g}= \sqrt{2}$. But now we 
can  just apply Theorem \ref{summit} to  $(\hat{M},\varpi^*g)$, thereby showing that it
is a K\"ahler manifold with K\"ahler form $\omega$ and positive scalar curvature. In particular, this implies
that $b_+(\hat{M}) =1$. Thus,  any  self-dual  harmonic $2$-form on 
$(\hat{M}, \varpi^*g)$ is a constant multiple of the K\"ahler form $\omega$. 
However, by construction, there is an  involution $\sigma:  \hat{M}\to  \hat{M}$ 
with $\varpi\circ \sigma = \varpi$ and $\sigma^* \omega= -\omega$. It follows that 
$b_+(M)= 0$, since a non-trivial  self-dual harmonic form on $(M,g)$ would otherwise 
pull back to a $\sigma$-invariant self-dual harmonic form on $(\hat{M}, \varpi^*g )$; and this 
 is impossible, because  any such  form would  also have to be a constant multiple of  $\omega$, which is not $\sigma$-invariant. 
We have thus  proved the following: 

\begin{prop} 
\label{mule}
Let $(M,h)$ be a compact  oriented Riemannian $4$-manifold with 
$\delta W^+ =0$ and $\det (W^+) > 0$. Then either
\begin{enumerate}[{\rm (i)}]
\item $b_+(M)=1$, and there is an orientation-compatible K\"ahler metric $g$ on $M$ of scalar curvature $s>0$, 
such that $h=s^{-2}g$; or else
\item $b_+(M)=0$, and there is a conformal rescaling $g$ of $h$ whose pull-back $\varpi^*g$  to a suitable double cover $\varpi: \hat{M}\to M$ 
 is a positive-scalar curvature K\"ahler metric on $\hat{M}$ that  is related to 
$\varpi^*h$  as in case {\rm (i)}. 
\end{enumerate}
\end{prop} 

Theorem \ref{whoa} is now  an immediate consequence of   Proposition \ref{mule}.

Notice that the  conformally rescaled metric $g$ is globally well-defined on $M$
in both cases of Propostion \ref{mule}; moreover, it  has 
scalar curvature $s> 0$, and may be renormalized   so as to arrange that $h=s^{-2}g$. The distinction between the two 
cases  is really a matter of holonomy; in the first case, the holonomy of $g$ is a subgroup of $\mathbf{U(2)}$, 
while in the second case it  instead   belongs to  the  larger  group $\mathbf{U(2)}\rtimes \ZZ_2$ of real-linear transformations of $\CC^2$
generated by  complex conjugation $(z^1,z^2)\mapsto (\bar{z}^1,\bar{z}^2)$ and   the unitary  transformations. Of course, 
the natural representation $\mathbf{U(2)}\rtimes \ZZ_2\to \ZZ_2$ gives rise to a double cover $\hat{M}\to M$,  and passing to this cover
then  simplifies matters  by reducing to the case of $\mathbf{U}(2)$ holonomy.  

 Using \cite{lebhem}  and the simple-connectivity of del Pezzos, we also now have: 
\begin{prop} 
\label{rule}
Let $(M,h)$ be a compact  oriented Riemannian Einstein $4$-manifold with 
$\det (W^+) > 0$. Then $(M,h)$ actually satisfies $\det (W^+) > 0$ at every point. Moreover, either
\begin{enumerate}[{\rm (i)}]
\item $\pi_1(M)=0$, and $M$ admits an orientation-compatible  complex structure $J$ that makes $(M,J)$  into a del Pezzo surface, and 
relative to which  the Einstein metric $h$ becomes conformally K\"ahler; or else, 
\item $\pi_1(M)=\ZZ_2$, and $M$ is doubly covered by a del Pezzo surface   $(\hat{M},J)$ of even signature on which 
 the pull-back of the Einstein metric $h$ becomes conformally K\"ahler.   \end{enumerate}
\end{prop}

Theorem \ref{whoo}  now becomes   an immediate corollary of Proposition \ref{rule}.

%
%To prove our result,  we now only need to make a  choice of the conformal factor $f$ that is correctly   adapted to our problem. 
%To do so, let us first observe how $\alpha= \alpha_g$ depends on the choice of $f$. Now recall that 
%${W^{+a}}_{bcd}$ has conformal weight zero, but that 
%$$[W^+(\varphi)]_{cd} := \frac{1}{2} {W^{+ab}}_{cd}~\varphi_{ab}$$
%for any self-dual $2$-form $\varphi$. Taking $\varphi$ to belong to the top eigenspace of $W^+$, 
%we thus see that  $h= f^2 g$ implies 
%$$\alpha_h = f^{-2}\alpha_g.$$
%Thus, if we choose to set
%$$\boxed{f= \alpha_h^{-1/3}}$$
%we will then have 
%$$\alpha := \alpha_g = f^2\alpha_h= \alpha_h^{1/3}= f^{-1},$$
%and hence  $\alpha f\equiv 1$. 
%For this choice, the above calculation then simplifies to yield 
%\begin{eqnarray*} 0 &\geq& 
% \frac{1}{2} \int_M |\nabla\omega |^2 ~ d\mu  +
% \frac{3}{2}  \int_M 
%\Big\langle \omega , (d+d^*)^2 \omega \Big\rangle 
% ~d\mu  \\ 
% &=& 
% \frac{1}{2} \int_M |\nabla\omega |^2 ~ d\mu  +
% 3 \int_M |d\omega|^2 
% ~d\mu  \\ 
% &\geq& 
% \frac{1}{2} \int_M |\nabla\omega |^2 ~ d\mu
% \end{eqnarray*}

\section{The Proof of Theorem \ref{cadenza}}
\label{grist}

The method used to prove  Theorems \ref{whoo} and \ref{whoa} does not actually require $\det W^+$ to be positive. Indeed, 
in this section, we will obtain essentially the same results under the weaker assumption that   the top and middle eigenvalues $\alpha$ and  $\beta$ of $W^+$ satisfy 
$$4\beta \leq {\alpha} \neq 0$$
everywhere. The following lemma will allow us to restate this hypothesis as an effective condition on $\det (W^+)$.

\begin{lem} 
\label{glissando}
Let $(M,h)$ be an oriented  Riemannian $4$-manifold, and let $p\in M$ be a point where $W^+\neq 0$. Let $\alpha$
and $\beta$ once again denote the highest and middle eigenvalues of $W^+$ at $p$. Then 
$$\beta \leq \frac{\alpha}{4} \quad \Longleftrightarrow \quad \det (W^+) \geq -{\textstyle \frac{5}{21} \sqrt{\frac{2}{21}}}~ |W^+|^3.
$$
Moreover, both of these equivalent  statements are  conformally invariant, in the sense that if either holds at $p$ for 
the metric $h$, then both necessarily hold at $p$ for every metric $g$ which is a conformal rescaling of   $h$. 
\end{lem}
\begin{proof}Let $x=\beta/\alpha$, and then notice that,  because  $\alpha \geq \beta \geq -\alpha -\beta$, 
we automatically have $x\in [-\frac{1}{2},1]$. Now set $y= 1+x+x^2$, and notice that $x\mapsto y$ defines an
increasing smooth map 
$[-\frac{1}{2},1]\to [\frac{3}{4} , 3]$ because $\frac{dy}{dx} = 1+ 2x$ is non-negative for $x\geq -\frac{1}{2}$. 
But this now makes it   apparent that 
\begin{eqnarray*}
\frac{\det W^+}{|W^+|^3} &=& \frac{\alpha \beta (-\alpha - \beta) }{(\alpha^2 + \beta^2 + (-\alpha - \beta)^2)^{3/2}}
\\ &=& -\frac{x +x^2}{2^{3/2}(1+x+x^2)^{3/2}}
\\ &=&-2^{-3/2}\left(y^{-1/2}-y^{-3/2} \right) 
\end{eqnarray*}
is a decreasing function of $x\in [-\frac{1}{2},1]$,
since 
\begin{eqnarray*}
\frac{d}{dy} \left[-2^{-3/2}\left(y^{-1/2}-y^{-3/2} \right) \right] &=&-2^{-3/2}\left( -\frac{1}{2}y^{-3/2}+ \frac{3}{2}y^{-5/2} \right) \\&=& -(2y)^{-5/2}(  3 -y)
\end{eqnarray*}
is non-positive for $y\in [\frac{3}{4} , 3]$. As a consequence, 
$$
\frac{\beta}{\alpha}\leq \frac{1}{4} \quad \Longleftrightarrow \quad \frac{\det W^+}{|W^+|^3}\geq -\frac{x +x^2}{2^{3/2}(1+x+x^2)^{3/2}}\Big|_{x=\frac{1}{4}}= -\frac{5\sqrt{2}}{21\sqrt{21}}.
$$
Moreover, since both $\det (W^+)/|W^+|^3$ and $x=\beta/\alpha$ are manifestly unaltered by conformal changes of the metric,  the equivalence in
question is obviously conformally invariant. 
\end{proof}

Most of the ideas we used in \S \ref{gist} merely depend on the assumption that the top eigenvalue $\alpha$
of $W^+$ has multiplicity one everywhere. However,  the key inequality \eqref{negatron}   is  quite  different, and  strongly  depends on the 
 assumption that $\det (W^+) > 0$. Nonetheless, we can generalize this inequality as follows:

\begin{lem}
\label{arpeggio} Let $(M,h)$ be an oriented  Riemannian $4$-manifold on which  the 
top eigenvalue $\alpha_h$ of $W^+_h$ has multiplicity one everywhere, and so defines a smooth function $\alpha_h$ on 
$M$.   Also suppose that the top eigenspace $L\subset \Lambda^+$ of
$W^+$ is trivial as  real line bundle $L\to M$. Let $g=f^{-2}h$ be some conformal rescaling of $h$,   and  let $\omega$ then  be a self-dual $2$-form 
on $M$  that  satisfies \eqref{charlie} everywhere. 
Let  $\beta= \beta_g:M\to \RR$ be the  continuous function given by the 
middle eigenvalue of $W^+_g$ at each point of $M$. 
Then 
\begin{equation}
\label{positron}
W^+ (\nabla_e\omega, \nabla^e\omega) \leq \beta |\nabla \omega |^2
\end{equation}
everywhere, where all terms are to be computed  with respect to $g$. 
\end{lem}
\begin{proof} The covariant derivative   $\nabla \omega$ of $\omega$ belongs to $\Lambda^1 \otimes \omega^\perp \subset\Lambda^1 \otimes  \Lambda^+$ because
$\omega$ has constant norm with respect to $g$. The result therefore follows from the fact that $W^+(\phi , \phi) \leq \beta |\phi |^2$ for any 
$\phi \in \omega^\perp$. 
\end{proof} 

With these lemmata in hand,  a return  visit to  our previously-explored  territory immediately  reveals the following:  

\begin{thm} 
\label{vista}
Let $(M,h)$ be a compact oriented Riemannian manifold with $\delta W^+=0$. Assume that 
$W^+\neq 0$ everywhere, and  that 
$$\det (W^+) \geq -{\textstyle \frac{5}{21} \sqrt{\frac{2}{21}}}~ |W^+|^3$$
at every point. Noting that this in particular implies  that the top eigenvalue $\alpha_h: M\to \RR^+$ of $W^+_h$ defines a smooth positive function on $M$, let us also  
now suppose  that the real line bundle $L\to M$ given by the $\alpha_h$-eigenspace of $W^+_h$ is trivial. Then the conformally rescaled metric
$$g = \alpha_h^{2/3}h$$
is an orientation-compatible  K\"ahler metric of positive scalar curvature.
\end{thm}
\begin{proof} Let $\omega$ be a self-dual $2$-form  that satisfies \eqref{charlie} at every point of
$M$ with respect to the rescaled metric $g$. Here we have once again arranged that $g=f^{-2}h$ 
has the property that $\alpha:=\alpha_g$ satisfies 
$$\alpha f \equiv 1,$$
as in \eqref{cunning},  by choosing $f$ according to \eqref{shrewd}.
Now Lemma \ref{glissando} tells us that our hypotheses imply that $\beta \leq \alpha/4$, while
Lemma \ref{arpeggio} provides us with a crucial   inequality \eqref{positron}  involving $\beta$. 
On the other hand, \eqref{megatron} and \eqref{harmless} are completely general facts about 
$4$-dimensional geometry that in particular apply to our current situation. Assembling  these pieces, 
we therefore have 
\begin{eqnarray*} 0&=& 
\int_M \Big\langle\Big( \nabla^*\nabla fW^+ + \frac{s}{2} fW^+ - 6 fW^+\circ W^+ + 2 f|W^+|^2 I \Big), \omega \otimes
\omega \Big\rangle ~d\mu_g \\
&=& \int_M \Big[ \langle W^+ , \nabla^*\nabla (\omega\otimes \omega )\rangle  + \frac{s}{2} W^+(\omega , \omega ) 
 - 6 |W^+(\omega)|^2+ 2 |W^+|^2 |\omega |^2 \Big] f~ d\mu_g  \\
&=&  
\int_M \Big[ 
- 2W^+(\nabla_e \omega , \nabla^e \omega )-2W^+(\omega , \nabla^e\nabla_e \omega  ) 
\\&& \hspace{2in}  + \frac{s}{2} \alpha |\omega|^2 
 - 6 \alpha^2 |\omega|^2+ 2 |W^+|^2 |\omega |^2 \Big] f~ d\mu_g  \\
 &\geq &  
\int_M \Big[ -2\beta  | \nabla \omega|^2 
-2\alpha \langle \omega , \nabla^e\nabla_e \omega  \rangle  
+ \frac{s}{2} \alpha |\omega|^2 
 - 6 \alpha^2 |\omega|^2+ 3 \alpha^2 |\omega |^2 \Big] f~ d\mu_g  \\
 &\geq &  
\int_M \Big[ -\frac{\alpha}{2}  | \nabla \omega|^2 
-2\alpha \langle \omega , \nabla^e\nabla_e \omega  \rangle  
+ \frac{s}{2} \alpha |\omega|^2 
 -  3 \alpha^2 |\omega |^2 \Big] f~ d\mu_g  \\
  &=&  
\int_M \Big[ 
-\frac{1}{2}  | \nabla \omega|^2 + 2\langle \omega , \nabla^*\nabla \omega  \rangle  
+ \frac{s}{2}  |\omega|^2 
 - 3 \alpha |\omega|^2 \Big] ~(\alpha f)~ d\mu_g  \\
 &=&  
\int_M \Big[ 
\frac{3}{2}\langle \omega , \nabla^*\nabla \omega  \rangle 
+ \frac{s}{2}  |\omega|^2 
 - 3 W^+(\omega , \omega) \Big] ~ d\mu_g  \\ 
 &=&   \frac{3}{2} 
\int_M \Big\langle \omega , \nabla^*\nabla \omega
- 2 W^+(\omega) + \frac{s}{3}\omega 
 \Big\rangle 
 ~ d\mu_g  \\ 
 &=& 
 \frac{3}{2}  \int_M 
\Big\langle \omega , (d+d^*)^2 \omega \Big\rangle 
~ d\mu_g  \\ 
 &=& 
 3 \int_M |d\omega|^2 
 ~d\mu   , 
 \end{eqnarray*}
 so the self-dual $2$-form $\omega$ must actually be closed, and hence harmonic. However, since $\omega$  also  has constant norm $\sqrt{2}$,
 this means that   $(M^4,g,\omega)$ is an almost-K\"ahler manifold. But, by construction,   $W^+(\omega , \omega ) > 0$  
 and    $h=f^2g$ satisfies $\delta (W^+)=0$. Thus, by   \cite[Proposition 2]{lebcake}, our almost-K\"ahler manifold  is actually  K\"ahler,
 and has positive 
 scalar curvature.
 \end{proof}
 
 However,  a K\"ahler surface $(M,g,J)$ of positive scalar curvature 
 necessarily satisfies $\det (W^+)>0$ at every point. Moreover, a result of Yau \cite{yauruled}
 guarantees that any such $(M,g,J)$ must have vanishing geometric genus, and so enjoys the topological property that $b_+(M)=1$.
  Applying Theorem \ref{vista} either to  $M$ or to the double cover  $\hat{M}\to M$ associated with the real line bundle $L$, the same argument 
  used to prove   Proposition \ref{mule} now  yields the following:

\begin{prop} 
Let $(M,h)$ be a compact  oriented Riemannian $4$-manifold with 
$\delta W^+ =0$ that also satisfies 
$$\det (W^+) \geq -{\textstyle \frac{5}{21} \sqrt{\frac{2}{21}}}~ |W^+|^3$$
at every point. Then actually $\det (W^+)>0$ everywhere, and either
\begin{enumerate}[{\rm (i)}]
\item $b_+(M)=1$, and there is an orientation-compatible K\"ahler metric $g$ on $M$ of scalar curvature $s>0$, 
such that $h=s^{-2}g$; or else
\item $b_+(M)=0$, and there is a conformal rescaling $g$ of $h$ whose pull-back $\varpi^*g$  to a suitable double cover $\varpi: \hat{M}\to M$ 
is  a positive-scalar curvature K\"ahler metric on $\hat{M}$ that  is related to 
$\varpi^*h$  as in case {\rm (i)}. 
\end{enumerate}
\end{prop} 

Similarly, the same reasoning used to prove Proposition \ref{rule} now yields: 

\begin{prop} 
Let $(M,h)$ be a compact  oriented Einstein $4$-manifold 
 that also satisfies 
$$\det (W^+) \geq -{\textstyle \frac{5}{21} \sqrt{\frac{2}{21}}}~ |W^+|^3$$
at every point.  Then $(M,h)$ satisfies $\det (W^+) > 0$ everywhere, and either
\begin{enumerate}[{\rm (i)}]
\item $\pi_1(M)=0$, and $M$ admits an orientation-compatible  complex structure $J$ that makes $(M,J)$  into a del Pezzo surface, and 
relative to which  the Einstein metric $h$ becomes conformally K\"ahler; or else, 
\item $\pi_1(M)=\ZZ_2$, and $M$ is doubly covered by a del Pezzo surface   $(\hat{M},J)$ of even signature on which 
 the pull-back of the Einstein metric $h$ becomes conformally K\"ahler.   \end{enumerate}
\end{prop}

Theorem \ref{cadenza} is now an immediate corollary of these last Propositions.

\bigskip

\noindent
{\bf Acknowledgments.} The author would like to thank Peng Wu for bringing Theorem \ref{whoo} to his attention.  This work was
carried out while the author was on sabbatical as a Simons Fellow in Mathematics, and was also supported in part by 
NSF grant DMS-1906267.
 
\pagebreak 

%
%\bibliographystyle{siam}
%\bibliography{lebrun}

\begin{thebibliography}{10}

\bibitem{bpv}
{\sc W.~Barth, C.~Peters, and A.~Van~de Ven}, {\em Compact Complex Surfaces},
  vol.~4 of Ergebnisse der Mathematik und ihrer Grenzgebiete (3),
  Springer-Verlag, Berlin, 1984.

\bibitem{bergb}
{\sc M.~Berger}, {\em Sur les vari\'et\'es d'{E}instein compactes}, in Comptes
  {R}endus de la {III}e {R}\'eunion du {G}roupement des {M}ath\'ematiciens
  d'{E}xpression {L}atine ({N}amur, 1965), Librairie Universitaire, Louvain,
  1966, pp.~35--55.

\bibitem{bes}
{\sc A.~L. Besse}, {\em Einstein Manifolds}, vol.~10 of Ergebnisse der
  Mathematik und ihrer Grenzgebiete (3), Springer-Verlag, Berlin, 1987.

\bibitem{bcg}
{\sc G.~Besson, G.~Courtois, and S.~Gallot}, {\em Entropies et rigidit{\'e}s
  des espaces localement sym{\'e}triques de courbure strictement n{\'e}gative},
  Geom. and Func. An., 5 (1995), pp.~731--799.

\bibitem{bohm}
{\sc C.~B{\"o}hm}, {\em Inhomogeneous {E}instein metrics on low-dimensional
  spheres and other low-dimensional spaces}, Invent. Math., 134 (1998),
  pp.~145--176.

\bibitem{bgk}
{\sc C.~P. Boyer, K.~Galicki, and J.~Koll{\'a}r}, {\em Einstein metrics on
  spheres}, Ann. of Math. (2), 162 (2005), pp.~557--580.

\bibitem{chenlebweb}
{\sc X.~X. Chen, C.~LeBrun, and B.~Weber}, {\em On conformally {K}\"ahler,
  {E}instein manifolds}, J. Amer. Math. Soc., 21 (2008), pp.~1137--1168.

\bibitem{derd}
{\sc A.~Derdzi{\'n}ski}, {\em Self-dual {K}\"ahler manifolds and {E}instein
  manifolds of dimension four}, Compositio Math., 49 (1983), pp.~405--433.

\bibitem{G1}
{\sc M.~J. Gursky}, {\em Four-manifolds with $\delta {W}\sp +=0$ and {E}instein
  constants of the sphere}, Math. Ann., 318 (2000), pp.~417--431.

\bibitem{hit}
{\sc N.~J. Hitchin}, {\em Compact four-dimensional {E}instein manifolds}, J.
  Differential Geometry, 9 (1974), pp.~435--441.

\bibitem{hitpos}
\leavevmode\vrule height 2pt depth -1.6pt width 23pt, {\em On the curvature of
  rational surfaces}, in Differential geometry (Proc. Sympos. Pure Math., Vol.
  XXVII, Part 2, Stanford Univ., Stanford, Calif., 1973), Amer. Math. Soc.,
  Providence, R. I., 1975, pp.~65--80.

\bibitem{lmo}
{\sc C.~LeBrun}, {\em Einstein metrics and {M}ostow rigidity}, Math. Res.
  Lett., 2 (1995), pp.~1--8.

\bibitem{lebhem}
\leavevmode\vrule height 2pt depth -1.6pt width 23pt, {\em Einstein metrics on
  complex surfaces}, in Geometry and {P}hysics ({A}arhus, 1995), vol.~184 of
  Lecture Notes in Pure and Appl. Math., Dekker, New York, 1997, pp.~167--176.

\bibitem{lebuniq}
\leavevmode\vrule height 2pt depth -1.6pt width 23pt, {\em On {E}instein,
  {H}ermitian 4-manifolds}, J. Differential Geom., 90 (2012), pp.~277--302.

\bibitem{lebhem10}
\leavevmode\vrule height 2pt depth -1.6pt width 23pt, {\em Einstein manifolds
  and extremal {K}\"ahler metrics}, J. Reine Angew. Math., 678 (2013),
  pp.~69--94.

\bibitem{lebcake}
\leavevmode\vrule height 2pt depth -1.6pt width 23pt, {\em Einstein metrics,
  harmonic forms, and symplectic four-manifolds}, Ann. Global Anal. Geom., 48
  (2015), pp.~75--85.

\bibitem{sunspot}
{\sc Y.~Odaka, C.~Spotti, and S.~Sun}, {\em Compact moduli spaces of del
  {P}ezzo surfaces and {K}\"ahler-{E}instein metrics}, J. Differential Geom.,
  102 (2016), pp.~127--172.

\bibitem{pr2}
{\sc R.~Penrose and W.~Rindler}, {\em Spinors and space-time. {V}ol. 2},
  Cambridge University Press, Cambridge, 1986.
\newblock Spinor and twistor methods in space-time geometry.

\bibitem{mhsung}
{\sc M.-H. Sung}, {\em K\"{a}hler surfaces of positive scalar curvature}, Ann.
  Global Anal. Geom., 15 (1997), pp.~509--518.

\bibitem{tian}
{\sc G.~Tian}, {\em On {C}alabi's conjecture for complex surfaces with positive
  first {C}hern class}, Invent. Math., 101 (1990), pp.~101--172.

\bibitem{wazi}
{\sc M.~Y. Wang and W.~Ziller}, {\em Einstein metrics on principal torus
  bundles}, J. Differential Geom., 31 (1990), pp.~215--248.

\bibitem{pengwu}
{\sc P.~Wu}, {\em Einstein four-manifolds of nonnegative determinant half
  {W}eyl curvature}.
\newblock e-print {\tt arXiv:1903.11818 [math.DG]}.

\bibitem{yauruled}
{\sc S.~T. Yau}, {\em On the curvature of compact {H}ermitian manifolds},
  Invent. Math., 25 (1974), pp.~213--239.

\end{thebibliography}

\vfill 

\noindent 
{\sc Department of Mathematics, 
Stony Brook University (SUNY),
%State University of New York, 
Stony Brook, NY 11794-3651 USA} 

\medskip 

\noindent 
{e-mail:} claude@math.sunyb.edu

\bigskip 

\noindent 
{\sc Keywords:} Einstein metric, Del Pezzo surface, Weyl curvature, K\"ahler,  conformal,  self-dual.

\bigskip 

\noindent 
{\sc MSC classification:}  53C25 (Primary),  14J26,  32J15, 53C55.

\end{document}